\documentclass{article}

\usepackage{amsmath,amssymb, amsthm, geometry, graphicx, wrapfig}

\theoremstyle{plain}
    
    \newtheorem{theorem}{Theorem}
    \newtheorem{statement}{Proposition}

  \theoremstyle{definition}
    \newtheorem{definition}{Definition}
        \newtheorem{remark}{Remark}
                   \newtheorem{example}{Example}

\usepackage[numbers, sort&compress]{natbib}

\newcommand{\rank}[1]{\mathrm{rank} \, #1}

\newcommand{\diff}[1]{\mathrm{d}  #1}

\newcommand{\diffFX}[2]{ \dfrac{\partial #1}{\partial #2} }
\newcommand{\diffFXi}[2]{ {\partial #1} / {\partial #2} }
\newcommand{\diffX}[1]{ \frac{\partial }{\partial #1} }
\newcommand{\diffXi}[1]{ {\partial } / {\partial #1} }

\newcommand{\R}{\mathbb{R}}
\newcommand{\Complex}{\mathbb{C}}

\newcommand{\T}{\mathrm{T}}
\newcommand{\PP}{\mathcal{P}}
\newcommand{\QQ}{\mathcal{Q}}
\newcommand{\Cont}{\mathrm{C}}

\newcommand{\CP}{\mathbb{C}\mathrm{P}}

\newcommand{\tr}[1]{\mathrm{tr} \, #1}
\newcommand{\eps}{\varepsilon}

\newcommand{\PoissonBracket}{ \{ \, , \} }
\newcommand{\g}{\mathfrak{g}}

\newcommand{\GL}{\mathrm{GL}}

\newcommand{\ad}{\mathrm{ad}\,}
\newcommand{\adi}[2]{\mathrm{ad}_{#1} \, #2}

\newcommand{\sgrad}[2]{\mathrm{sgrad}_{#1} \, #2}
\newcommand{\divergence}[1]{\mathrm{div} \, #1}

\title{Curvature of Poisson pencils in dimension three}
\author{Anton Izosimov\footnote{Moscow State University. e-mail: izosimov@mech.math.msu.su}}
\date{}

\begin{document}
\maketitle
\begin{abstract}
A Poisson pencil is called flat if all brackets of the pencil can be simultaneously locally brought to a constant form. Given a Poisson pencil on a $3$-manifold, we study under which conditions it is flat. Since the works of Gelfand and Zakharevich, it is known that a pencil is flat if and only if the associated Veronese web is trivial. We suggest a simpler obstruction to flatness, which we call the curvature form of a Poisson pencil. This form can be defined in two ways: either via the Blaschke curvature form of the associated web, or via the Ricci tensor of a connection compatible with the pencil.\par
We show that the curvature form of a Poisson pencil can be given by a simple explicit formula. This allows us to study flatness of linear pencils on three-dimensional Lie algebras, in particular those related to the argument translation method. Many of them appear to be non-flat.
\end{abstract}
\centerline{MSC: 37K10, 53D17, 53A60}
\section{Introduction}
Two Poisson brackets are called \textit{compatible}, if any linear combination of them is a Poisson bracket again. The notion of compatible Poisson brackets was introduced by F.Magri \cite{Magri} and I.Gelfand and I.Dorfman \cite{GD} because of its relation to integrability. Roughly speaking, if a dynamical systems is hamiltonian with respect to two compatible Poisson structures (i.e. it is \textit{bi-Hamiltonian}), then it automatically possesses many conservation laws. This mechanism is responsible for the integrability of many equations coming from physics and geometry (see \cite{Dorfman, Blaszak} and references therein).\par
A pair of two compatible Poisson brackets is called a \textit{Poisson pair}. Equivalently, one may speak about \textit{Poisson pencils}. A Poisson pencil is the set of all linear combinations of two compatible brackets.\par
Unlike individual Poisson brackets, which can be always locally brought to a constant form, Poisson pencils have non-trivial geometry. Differential geometry of Poisson pencils was studied, among others, by I.Gelfand and I.Zakharevich \cite{GZ, GZ2, Zakharevich, GZ3}, F.-J.Turiel \cite{Turiel, Turiel2, Turiel3, Turiel4}, and A.Panasyuk \cite{Panasyuk}.\par
Speaking about Poisson pencils, one needs to distinguish between the even and the odd-dimensional cases. The reason for this comes from linear algebra. A generic skew-symmetric form on an even-dimensional vector space is non-degenerate. So, for two forms $\PP$ and $\QQ$ one may define the operator $\PP\QQ^{-1}$. The eigenvalues of this operator are invariants of the pair $(\PP,\QQ)$. In odd dimension such invariants do not exist. Any two generic pairs of forms in this case are equivalent.\par
Consequently, in odd dimension any two generic Poisson pencils are equivalent at a point. However, this is no longer the case in the neighbourhood of a point. This observation makes odd-dimensional bi-Poisson geometry similar to Riemannian geometry. \par
A Poisson pencil is called \textit{flat} if all brackets of the pencil can be simultaneously locally brought to a constant form (like flat metrics). We are interested in the following problem. Given a generic Poisson pencil in odd dimension, how do we verify its flatness?\par
 This problem was intensively studied by I.Gelfand and I.Zakharevich \cite{GZ, GZ2, Zakharevich, GZ3}. For each generic Poisson pencil in odd dimension, Gelfand and Zakharevich construct a Veronese web\footnote{Or a $d+1$-web in $\R^d$, which is almost the same, see \cite{Dufour}.} naturally associated with the pencil and prove that the pencil is flat if and only if the associated web is trivial. This, in principle, allows to verify flatness for any given pencil. However,  it is difficult in practice. Passing from a pencil to the associated web involves the calculation of Casimir functions. To find Casimir functions of a Poisson bracket, one needs to solve partial differential equations. These equations are not always soluble by quadratures. This means that the explicit description of the Veronese web associated with a Poisson pencil is, in general, not possible.\par
So, it would be desirable to construct a curvature-like obstruction to flatness of a Poisson pencil. We do this in dimension three. The obstruction turns out to be a $2$-form. We call it \textit{the curvature form of a Poisson pencil}. This form can be defined in two ways: either via the Blaschke curvature form of the associated web, or via the Ricci tensor of a connection compatible with the pencil.\par
We show that the curvature form of a Poisson pencil can be given by a simple explicit formula. This allows us to study flatness of linear pencils on three-dimensional Lie algebras, in particular those related to the argument translation method. Many of them appear to be non-flat.

\section{Basics of bi-Hamiltonian geometry}
Throughout the paper, all objects belong to the class $\Cont^\infty$.
\begin{definition}
	Two Poisson brackets on a manifold $M$ are called \textit{compatible}, if any linear combination of them is a Poisson bracket again.  A pair of compatible brackets is called a \textit{Poisson pair}.
\end{definition}
For two Poisson brackets to be compatible, it is enough to require that their sum is also a Poisson bracket.
\begin{definition}
Let $(\PP,\QQ)$ be a Poisson pair. The set
$$
\{ \alpha \PP + \beta \QQ \mid \alpha, \beta \in \Complex \}
$$ is called the \textit{Poisson pencil} generated by $\PP,\QQ$.
\end{definition}
In other words, a Poisson pencil is a two-dimensional vector space of Poisson brackets. Choosing a basis in a Poisson pencil, one obtains a Poisson pair.
\begin{remark}
Further we use the following notation. Poisson brackets are denoted by $\PP, \QQ$ being treated as tensors, and $\PoissonBracket_{\mathbf p}, \PoissonBracket_{\mathbf q}$ being treated as operations on functions.
\end{remark}
\begin{definition}\label{def1}
 The rank of a Poisson pencil $\Pi = \{ \alpha \PP + \beta \QQ\}$ at a point $x$ is the number
 $$\rank \Pi(x) = \max_{\alpha, \beta} \rank (\alpha \PP(x) + \beta \QQ(x)).$$
 \end{definition}
  \begin{definition}\label{def2}
 The spectrum of a Poisson pencil $\Pi = \{ \alpha \PP + \beta \QQ\}$ at a point $x$ is the set
 $$
 \Lambda(x) = \{ (\alpha : \beta) \in \CP^1 \mid \rank (\alpha \PP(x) + \beta \QQ(x)) < \rank \Pi(x) \}.
 $$
\end{definition}
In dimension three, the spectrum is either empty or contains one element $\lambda$. In the latter case $\PP$ and $\QQ$ are proportional at $x$: $\PP(x) = -\lambda \QQ(x)$.
 \begin{definition}\label{def3}
A Poisson pencil $\Pi$ is said to be Kronecker at $x$ if its spectrum at $x$ is empty.
\end{definition}
In dimension three, this condition means that $\PP$ and $\QQ$ are not proportional at $x$.
\begin{definition}
 A Poisson pencil is called locally \textit{flat} if all brackets of the pencil can be simultaneously brought to a constant form in the neighbourhood of a generic point.
\end{definition}
Clearly, the rank and the spectrum of a flat pencil are locally constant. In dimension three, this is possible in three cases:
\begin{enumerate}
\item $\PP$ and $\QQ$ are identically zero;
\item $\PP = \lambda \QQ$, where $\lambda$ is constant;
\item the pencil is Kronecker of rank two.
\end{enumerate}
The first two cases are trivial. So, further we only consider Kronecker pencils of rank two.  Further we call such pencils \textit{generic}.

\section{Basics of web geometry}
\begin{wrapfigure}{r}{125pt}
\includegraphics[scale=0.25]{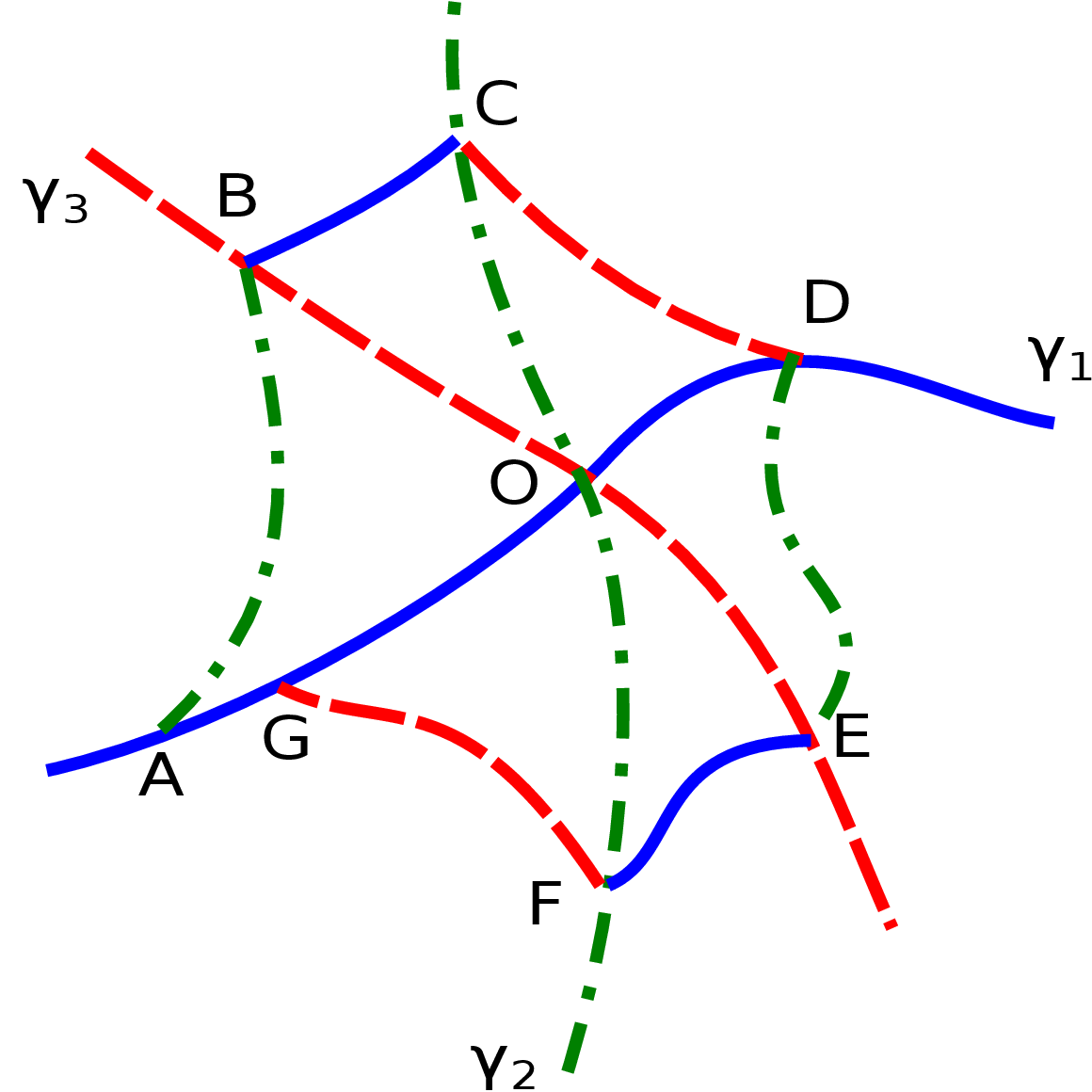}
\caption{$3$-Web}\label{web}
\end{wrapfigure}
\begin{definition}
	A $3$-web on a plane is three family of smooth curves such that
	\begin{enumerate}
		\item For any point $x$ there is a unique curve from each family passing through $x$.
		\item Curves from different families intersect transversally.
	\end{enumerate}
\end{definition}
\begin{definition} A $3$-web  is called trivial if it is diffeomorphic to a web given by three families of parallel lines. \end{definition}

Trivial $3$-webs are also called hexagonal $3$-webs due to the following construction. Take a point $O$. Let $\gamma_i$ be the curve from the $i$'th family passing through $O$. Take a point $A$ on $\gamma_1$ close to $O$ and consider the curve from the second family passing through $A$. This curve intersects $\gamma_3$ at a point $B$. Consider the curve from the first family passing through $B$. This curve intersects $\gamma_2$ at a point $C$. Continuing this procedure, we obtain a polygon $ABCDEFG$ depicted in Figure \ref{web}. In general, this polygon is not closed. If it is closed for any choice of $O$ and $A$, then the web is called hexagonal. Obviously, a trivial web is hexagonal. The inverse is also true.
\begin{theorem}[Blaschke \cite{Blaschke}]
A web is trivial if and only if it is hexagonal.
\end{theorem}
\clearpage
Now let us introduce the notion of curvature of a web. Suppose that the web is given by level sets of three functions $f_1, f_2, f_3$. Let $\eps = f_3(A) - f_3(O)$. Then it can be shown that $$f_3(G) - f_3(A) = \kappa\eps^3 + \dots,$$ so it makes sense to take $\kappa$ as a measure of ``curvature'' of a web. Of course, $\kappa$ depends on the choice of $f_3$. However, the differential form
$$
\theta = \kappa \, \diffFX{f_3}{f_1}\diffFX{f_3}{f_2}\,\diff f_1 \wedge \diff f_2
$$
depends only on the web itself. We will refer to $\theta$ as the \textit{Blaschke curvature form}.
\begin{theorem}[Blaschke \cite{Blaschke}]\label{BThm}
A web is trivial if and only if its curvature form is zero.
\end{theorem}
The following theorem allows to compute Blaschke curvature.
\begin{theorem}[Blaschke \cite{Blaschke}]\label{BFormula}
The curvature form of a web given by level sets of $f_1, f_2, f_3$ is
\begin{align*}
	\theta	= 2 \, \diffX{f_1}\diffX{f_2} \log \frac{\diffFXi{f_3}{f_2}}{\diffFXi{f_3}{f_1}} \, \diff f_1 \wedge \diff f_2.
\end{align*}
\end{theorem}

\section{Gelfand-Zakharevich reduction}\label{GZRed}
Consider a generic pencil $\Pi = \{ \alpha \PP + \beta \QQ\}$ on a $3$-manifold $M$. Fix a small ball $U \subset M$. Let $f,g,h$ be local Casimir functions of $\PP, \QQ, \PP+\QQ$ respectively such that $\diff f, \diff g, \diff h$ do not vanish.
\begin{statement}\label{p1}
$\diff f, \diff g, \diff h$ are pairwise linearly independent at every point.
\end{statement}
\begin{proof}
	Indeed, if two of them are dependant at some point, then the corresponding brackets have coinciding kernels. Two skew-symmetric forms in $\R^3$ with coinciding kernels are proportional. So, the pencil is not Kronecker.
\end{proof}
\begin{statement}\label{p2}
$\diff f, \diff g, \diff h$ are linearly dependent at every point.
\end{statement}
\begin{proof}
Since $f$ is a Casimir function of $\PP$, we have $
\{f,g\}_{\mathbf p} = 0.
$
Since $g$ is a Casimir function of $\QQ$, we have
$
\{f,g\}_{\mathbf q} = 0.
$
Therefore, $\{f,g\}_{\mathbf p+\mathbf q} = \{f,g\}_{\mathbf p} + \{f,g\}_{\mathbf q} = 0$. Since $h$ is a Casimir function of $\PP + \QQ$, we have $\{f,h\}_{\mathbf p+\mathbf q} = \{g,h\}_{\mathbf p+\mathbf q} = 0$. Consequently, $f,g,h$ commute with respect to $\PP+\QQ$. But since $\PP+\QQ \neq 0$, this is only possible if their differentials are linearly dependent.
\end{proof}

Propositions \ref{p1}, \ref{p2} imply that $h$ is a function of $f,g$. Now consider the map
\begin{align}\label{redMap}
\pi \colon U \to \R^2
\end{align}
given by $f,g$. Symplectic leaves of $\PP$ project under this map to level sets of $f$. Symplectic leaves of $\QQ$ project to level sets of $g$. Finally, symplectic leaves of $\PP+\QQ$ project to level sets of $h(f,g)$. Proposition \ref{p1} implies that these three families of curves form a $3$-web $W(\PP,\QQ)$.
\begin{definition}
The $3$-web $W(P,Q)$ is called the \textit{bi-Hamiltonian reduction} of the Poisson pair $(P,Q)$.
\end{definition}
\begin{theorem}[Gelfand-Zakharevich]\label{GZThm}
	A generic pencil $\{ \alpha \PP + \beta \QQ\}$ on a three-dimensional manifold is locally flat if and only if the web $W(\PP,\QQ)$ is trivial.
\end{theorem}
\begin{remark}
 This theorem is true in any dimension, provided that generic means Kronecker of corank one. In the analytic case it was proved by Gelfand and Zakharevich \cite{GZ,GZ2}. The proof in the $\Cont^{\infty}$ case is due to Turiel \cite{Turiel4}. Note that in the three-dimensional situation this result can be easily proved by elementary means.
\end{remark}

\section{Curvature form}
\begin{definition} Let $(\PP,\QQ)$ be a generic Poisson pair on a $3$-manifold. The curvature of $(\PP,\QQ)$ is
$$
\Theta(\PP,\QQ) = \pi^*(\theta),
$$
where $\pi$ is the reduction map (\ref{redMap}) and ${\theta}$ is the Blaschke curvature of the web $W(\PP,\QQ)$.
\end{definition}
\begin{theorem}\label{curvThm} Let $(\PP,\QQ)$ be a generic Poisson pair on a $3$-manifold. Then $\Theta(\PP,\QQ)$, being written in local coordinates $x,y,z$, reads
\begin{align}\label{curvFormula}
\Theta(\PP,\QQ) = 2\,\sum_\circlearrowleft\left( \,\left\{z,\frac{\divergence (\sgrad{\mathbf q}{z})}{\Delta_z}\right\}_{\!\!\mathbf p} - \left\{z,\frac{\divergence (\sgrad{\mathbf p}{z})}{\Delta_z}\right\}_{\!\!\mathbf q}\,\right) \diff x \wedge \diff y \, 
\end{align}
where
\begin{enumerate}
\item $\circlearrowleft$ denotes the sum over cyclic permutations of $x,y,z$;
\item $\sgrad{\mathbf p}{z}$ and $\sgrad{\mathbf q}{z}$ are the hamiltonian vector fields $\PP\diff z, \QQ\diff z$ respectively;
\item $\divergence V$ is the standard divergence $\sum \diffFXi{V^i}{x^i}$;
\item $\Delta_z$ is
\begin{align*}\Delta_z = \left|\,\vphantom{\int\limits_{a}{b}}\begin{matrix}
\left\{\vphantom{\frac{a}{b}}z,x\right\}_{\!\mathbf p} & \left\{\vphantom{\frac{a}{b}}z,x\right\}_{\!{\mathbf q}} \\
\left\{\vphantom{\frac{a}{b}}z,y\right\}_{\!{\mathbf p}} & \left\{\vphantom{\frac{a}{b}}z,y\right\}_{\!{\mathbf q}}
\end{matrix}\,\right|; \end{align*}
\item if $\Delta_x, \Delta_y$ or $\Delta_z$ is zero, then the respective term is omitted from the sum\footnote{Note that $\Delta_x, \Delta_y$ and $\Delta_z$ cannot be zero simultaneously, since $\PP$ and $\QQ$ are not proportional to each other.}.
\end{enumerate}
\end{theorem}
The proof can be found in Section \ref{proof}.
\begin{remark}
	Note that (\ref{BFormula}) involves third derivatives of $f_1, f_2, f_3$, while (\ref{curvFormula}) involves only first and second derivatives of the components of $\PP$ and $\QQ$.
\end{remark}
\begin{statement}
	If $\mathcal R$ and $\mathcal S$ are two non-proportional brackets from the pencil $\{ \alpha \PP + \beta \QQ\}$, then $\Theta(\mathcal R,\mathcal S) = \Theta(\PP,\QQ)$.
\end{statement}
\begin{proof}
	We need to show that the curvature does not change when we apply a non-degenerate linear transformation $(\PP,\QQ) \to (\mathcal R,\mathcal S)$. Formula (\ref{curvFormula}) implies that
	\begin{enumerate}
		\item $\Theta(\QQ,\PP) = \Theta(\PP,\QQ)$.
		\item $\Theta(\alpha \PP, \QQ) = \Theta(\PP,\QQ)$.
		\item $\Theta(\PP, \PP + \QQ) = \Theta(\PP, \QQ)$.	
	\end{enumerate}
	Now it suffices to note that these three transformations generate $\GL(2)$.
\end{proof}
So, the curvature of a Poisson pencil $\Pi$ is well-defined. We will denote it by $\Theta(\Pi)$.
The following result obviously follows from Theorems \ref{BThm} and \ref{GZThm}.
\begin{statement}
A generic Poisson pencil $\Pi$ on a $3$-manifold is locally flat if and only if $\Theta(\Pi) = 0$.
\end{statement}

\section{Torsion-free connection compatible with a Poisson pencil}
\begin{theorem}\label{conThm}
Let $\Pi$ be a generic pencil on a $3$-manifold. Then
\begin{enumerate}
\item Locally, there exists a (non-unique) torsion-free connection $\nabla$ compatible with the pencil, which means that
$\nabla \PP = 0$ for any $\PP \in \Pi$.
\item For any connection $\nabla$ with this property we have
\begin{align*}
	\Theta(\Pi) = -4\,\mathrm{Alt}\, \mathrm{Ric} \, \nabla,
\end{align*}
where $\mathrm{Ric} \,\nabla$ denotes the Ricci tensor and $\mathrm{Alt}$ means alternation.
\end{enumerate}
\end{theorem}

So, the curvature of a Poisson pencil can be defined using a torsion-free connection compatible with this pencil. The proof can be found in Section \ref{proof2}.

\begin{remark}
For any torsion-free connection, the first Bianchi identity implies that $$\mathrm{Alt}\, \mathrm{Ric}  = \frac{1}{2}\,\mathrm{Tr}\, \mathrm{R}$$
where $\mathrm R$ is the Riemann curvature tensor. So,
\begin{align*}
	\Theta(\Pi)(X,Y) = -2\,\mathrm{Tr}\, \mathrm{R}(X,Y).
\end{align*}
\end{remark}
\begin{remark}
 The skew-symmetric part of the Ricci tensor measures the deformation of infinitesimal volume by holonomy operators. So, if a connection is compatible with a non-degenerate bilinear form, such as a Riemannian metric, then its Ricci tensor must be symmetric. By this reason, the Ricci tensor of a Riemannian connection is always symmetric. However, the Ricci tensor of a connection compatible with a degenerate form, such as a Poisson structure on a $3$-manifold, is not necessarily symmetric.
\end{remark}
\begin{remark}
 By using a partition of unity, a global torsion-free connection compatible with a generic Poisson pencil can be constructed.
\end{remark}

\section{Curvature of linear pencils and the argument translation method}
\begin{definition}
	Let $\g$ be a Lie algebra and $\mathcal A$ be a skew-symmetric bilinear form on it.
	Then $\mathcal A$ can be considered as a Poisson tensor on the dual space $\g^{*}$. Assume that this tensor is compatible with the Lie-Poisson tensor. In this case, the Poisson pencil $\Pi(\g, \mathcal A)$ generated by these two tensors is called the \textit{linear pencil} associated with the pair $(\g, \mathcal A)$.
\end{definition}

The following is well known.
\begin{statement}
	A form $\mathcal A$ on $\g$ is compatible with the Lie-Poisson bracket if and only if this form is a Lie algebra $2$-cocycle, i.e.
	\begin{align*}
		\mathcal A([x,y],z) + \mathcal A([y,z],x) + \mathcal A([z,x],y) = 0
	\end{align*}
	for any $x, y, z \in \g$.
\end{statement}

\begin{example}
	Let $\g$ be a Lie algebra and $\xi \in \g^*$. Then $\mathcal A_\xi(x,y) = \langle \xi, [x,y]\rangle$ is a $2$-cocycle. 
The Poisson bracket defined by $\mathcal A_\xi$ on $\g^*$ is usually refereed to as a \textit{ ``frozen argument''} bracket. This bracket is related to the so-called \textit{argument translation method} \cite{MF, MF2} which appears in the theory of integrable systems.
	
\end{example}

For a linear pencil $\Pi(\g, \mathcal A)$, it is possible to rewrite (\ref{curvFormula}) in the following form.
\begin{align}\label{linCurvFormula}
\Theta(\Pi) = 2 \, \sum_\circlearrowleft \frac{\mathcal A(z, \Delta_z) \, \tr (\ad z)}{\Delta_z^2} \, \diff x \wedge \diff y \, 
\end{align}
where
\begin{enumerate}
\item $x,y,z$ is any basis in $\g$ (treated as a coordinate system in $\g^*$);
\item  $\circlearrowleft$ denotes the sum over cyclic permutations of $x,y,z$;
\item $\Delta_z$ is
\begin{align*}\Delta_z = \left|\,\vphantom{\int\limits_{a}{b}}\begin{matrix}
[z,x]& \mathcal A(z,x) \\
[z,y] & \mathcal A(z,y )\end{matrix}\,\right|; \end{align*}
\item $\ad z$ means the adjoint operator $\ad z = [z, *]$.
\end{enumerate}
\begin{remark}
Note that $\Delta_z$ is now a linear function on $\g^*$, i.e. an element of $\g$, so the expression $\mathcal A(z, \Delta_z)$ is well-defined.
\end{remark}
\begin{example}
	If $\g$ is semisimple, then $\tr \ad w = 0$ for any $w \in \g$. So,  $\Pi(\g, \mathcal A)$ is flat for any $\mathcal A$.
\end{example}
\begin{example}
Let $\g$ be a real Lie algebra given by 
\begin{align*}
	[z,x] = x,\, &[z,y] = ay,\\
	[x,y] &= 0
\end{align*}
and $\xi \in \g^*$. Consider the pencil $\Pi(\g, \mathcal A_\xi)$. This pencil is generic if $a \neq 0$ and at least one of the numbers $ \langle \xi,x \rangle,  \langle \xi,y \rangle$ is not zero. The curvature is given by
\begin{align}\label{curvExample1}
\Theta = \dfrac{2(1-a^2)\langle \xi,x \rangle\langle \xi,y \rangle}{a\left(\langle \xi,y \rangle x  - \langle \xi,x \rangle y\right)^2}\,\diff x \wedge \diff y.
\end{align}
So, for $a = \pm 1$ the pencil is flat for any $\xi$. For $a \neq \pm 1$ the pencil is not flat for generic $\xi$. However, it is flat if $\langle \xi,x \rangle = 0$ or $\langle \xi,y \rangle = 0$.\par
Note that the points in which the denominator of (\ref{curvExample1}) is zero are exactly those in which the pencil is not generic. This is always so for linear pencils. For non-linear pencils, this is not necessarily the case, see Example \ref{nonFlatLiePencil}.
\end{example}
\begin{remark}
	To author's knowledge, the only known example of a non-flat linear pencil was provided by A.\,Konyaev \cite{Konyaev}. The dimension of the corresponding Lie algebra is five.
\end{remark}
\begin{example}
Let $\g$ be a three-dimensional Lie algebra given by 
\begin{align*}
	[x,y] &= y,\\
	[z,x] = [&z,y] = 0.
\end{align*}
A form $\mathcal A$ on this algebra is a cocycle if $\mathcal A(y,z) = 0$. The pencil $\Pi(\g, \mathcal A)$  is generic if $\mathcal A(x,z) \neq 0$. The curvature is given by
$$
\Theta = \dfrac{2\,\mathcal A(x,y)}{\mathcal A(x,z)\,y^2}\,\diff y \wedge \diff z.
$$
\end{example}


%

\section{Curvature of Lie pencils}
\begin{definition}
	Two Lie algebra structures on a vector space $V$ are called \textit{compatible}, if any linear combination of them is a Lie algebra structure again. The set of all linear combinations of two compatible Lie structures is called a \textit{Lie pencil}.
\end{definition}
Analogously, we may say that two Lie structures on $V$ are compatible if they define compatible Poisson brackets on $V^*$.
Obviously, for two Lie brackets to be compatible, it is enough to require that their sum is also a Lie bracket.\par
Let $\PP,\QQ$ be two compatible Lie structures on $V$. We will denote the corresponding Poisson pencil on $V^*$ by $\Pi(V,\PP,\QQ)$.

\begin{remark}
	Poisson pencils related to Lie pencils also appear in integrable systems. In particular, the Manakov top \cite{Manakov} and the Clebsch top are bi-Hamiltonian with respect to a certain Lie pencil \cite{Bolsinov}.
\end{remark}

For a pencil $\Pi(V,\PP,\QQ)$, Formula (\ref{curvFormula}) can be rewritten as
\begin{align}\label{LieCurvFormula}
\Theta(\Pi) = 2\sum_\circlearrowleft\left(\,\left\{z,\frac{\tr (\adi{\mathbf q}{z})}{\Delta_z}\right\}_{\!\!\mathbf p} - \left\{z,\frac{\tr (\adi{\mathbf p}{z})}{\Delta_z}\right\}_{\!\!\mathbf q}\,\right) \, \diff x \wedge \diff y \,
\end{align}
where
\begin{enumerate}
\item $x,y,z$ is any basis in $V$ (treated as a coordinate system in $V^*$);
\item  $\circlearrowleft$ denotes the sum over cyclic permutations of $x,y,z$;
\item $\adi{\mathbf p}{z}$ and $\adi{\mathbf q}{z}$ denote the adjoint operators of the Lie structures $\PP$ and $\QQ$ respectively.
\end{enumerate}
\begin{example}
Suppose that $\PP$ is a semisimple Lie structure. Then $\Pi(V,\PP,\QQ)$ is flat. Indeed, if $\PP$ is semisimple, then $\PP + \eps \QQ$ is also semisimple for small $\eps$. Take $\PP$ and $\PP + \eps \QQ$ as generators of $\Pi$. Since $\tr \adi{\mathbf p}{z} = \tr \adi{\mathbf p+ \eps \mathbf q}{w} = 0$ for any $w \in V$, the curvature of $\Pi$ vanishes.
\end{example}
\begin{example}\label{nonFlatLiePencil}
A Lie Pencil is given by
\begin{align*}
\begin{matrix}
[z,x]_{\mathbf p} = x, & [z,x]_{\mathbf q} = 0,\\
[z,y]_{\mathbf p} = -y,& [z,y]_{\mathbf q} = x + y,\\
[x,y]_{\mathbf p} = 0, & [x,y]_{\mathbf q} = 0.
\end{matrix}
\end{align*}
Then
$$
\Theta = -\frac{4\,\diff x \wedge \diff y}{(x+y)^2},
$$
so the pencil is not flat. This can be proved without computing the curvature, see Example \ref{nonFlatLiePencil3}. Note that the pencil is not generic for $x=0$, however the curvature form can be continuously extended to the set $\{x = 0, y \neq 0\}$.
\end{example}
\section{Curvature and singularities}

It is not typical that a pencil is generic everywhere. As a rule, this condition is satisfied on an open and everywhere dense set. The complement to this set is the \textit{singular set} of the pencil, which we denote by $\Sigma$. \par
Let $\Pi$ be a pencil on a $3$-manifold $M$. First suppose that $x \in \Sigma$ and $\rank \Pi(x) = 2$. Then there exists a unique up to proportionality $\PP \in \Pi$ such that $\PP(x) = 0$. The linear part of the Poisson tensor $\PP$ defines a linear Poisson bracket on $\T_xM$. Now take $\QQ \in \Pi$ such that $\QQ(x) \neq 0$. Its restriction to $\T_xM$ defines a constant Poisson bracket on $\T_xM$. This bracket is compatible with the bracket given by the linear part of $\PP$. Thus we obtain a linear pencil on $\T_xM$. This pencil is called the \textit{linearization} of $\Pi$ at $x$.\par
Decomposing $\PP$ and $\QQ$ in Taylor series and applying (\ref{curvFormula}) we obtain the following.
\begin{statement}\label{flatLin}
If the linearization of a flat pencil is generic, then it is flat.
\end{statement}
\begin{example}\label{nonFlatLiePencil3}
The linearization  of the pencil from Example \ref{nonFlatLiePencil} at the point $(1,-1,0)$ is generic but not flat. So, the pencil is not flat.
\end{example}
The inverse statement to Proposition \ref{flatLin} is not true.
\begin{example}\label{nonFlatLiePencil2}
The linearization  of the pencil from Example \ref{nonFlatLiePencil} at the point $(0,1,0)$ is flat.
\end{example}

\begin{remark}
	The linearization defined above is a particular case of the so-called $\lambda$-\textit{linearziation} defined in \cite{SBS, JGP} for studying local behavior of bi-Hamiltonian systems. We suggest that a statement analogous to Proposition \ref{flatLin} is true for $\lambda$-linearization of a multidimensional pencil. 
\end{remark}

Now let $x \in \Sigma$ and $\rank \Pi(x) = 0$. Then all brackets of the pencil vanish at $x$. Taking linear parts of any two brackets, we obtain a Lie pencil. If this Lie pencil is generic and the initial pencil is flat, then the Lie pencil is also flat.

\section{Proof of Theorem \ref{curvThm}}\label{proof}
Let $f,g,h$ be local Casimir functions of $\PP, \QQ, \PP+\QQ$ respectively such that $\diff f, \diff g, \diff h$ do not vanish.
By Theorem \ref{BFormula}, the Blaschke curvature form of $W(\PP,\QQ)$ is given by
$$
	\theta	= \left(2\diffX{f}\diffX{g} \log \frac{\diffFXi{h}{g}}{\diffFXi{h}{f}}\right)\diff f \wedge \diff g.
	$$
We need to compute $\pi^*(\theta)$, where $\pi$ is given by $f,g$. Let us compute the $\diff x \wedge \diff y$ term. Two other terms are computed analogously.
Represent $\pi$ as the composition of two maps
$$
(x,y,z) \xrightarrow{\pi_1} (f,g,z) \xrightarrow{\pi_2} (f,g).
$$
We have $$\Theta(\PP, \QQ) = \pi^*(\theta) = \pi_1^*(\pi_2^*(\theta)).$$ Denote
$$
 \Delta_{fg} = \det \diff \pi_1= \det \begin{pmatrix}
\diffFXi{f}{x} & \diffFXi{g}{x} \\
\diffFXi{f}{y} & \diffFXi{g}{y}
\end{pmatrix}.
$$
Compute $\Delta_{fg}$.
 We have
$$
	\PP = \begin{pmatrix}
0 & \{x,y\}_{\mathbf p} & \{x,z\}_{\mathbf p} \\
\{y,x\}_{\mathbf p} & 0 & \{y,z\}_{\mathbf p} \\
\{z,x\}_{\mathbf p} & \{z,y\}_{\mathbf p} & 0
\end{pmatrix}.
$$
Since $f$ is a Casimir function of $\PP$, we have
$$
\diff f = \mu(\{y,z\}_{\mathbf p},\{z,x\}_{\mathbf p}, \{x,y\}_{\mathbf p}),
$$
where $\mu \neq 0$ is a function (an integrating factor). Analogously,
$$
\diff g = \nu(\{y,z\}_{\mathbf q},\{z,x\}_{\mathbf q} ,\{x,y\}_{\mathbf q}).
$$
Consequently,
$$
\Delta_{fg} = \mu \nu \Delta_z.
$$
\par First assume that $\Delta_z = 0$. Since $\pi_2^*(\theta)$ is proportional to $\diff f \wedge \diff g$, the curvature form $\Theta(\PP,\QQ) $ is proportional to $\diff f(x,y,z) \wedge \diff g(x,y,z)$. Further, $$\Delta_{fg} = \frac{\Delta_z}{\mu\nu} = 0,$$ so there is no $\diff x \wedge \diff y$ term in $\diff f(x,y,z) \wedge \diff g(x,y,z)$ and in $\Theta(\PP,\QQ)$, q.e.d.\par
Now assume that $\Delta_z \neq 0$. In this case $\Delta_{fg} \neq 0$, so $f,g,z$ is a well-defined local coordinate system. We have
$$\pi_2^*(\theta) = \left(2\diffX{f}\diffX{g} \log \frac{\diffFXi{h}{g}}{\diffFXi{h}{f}}\right)\diff f \wedge \diff g.$$
To get $\Theta(\PP,\QQ)$ we need to pass back to the coordinates $x,y,z$. The $\diff x \wedge \diff y$ term of $\Theta(\PP,\QQ)$ is

$$
\Theta_{xy} = \left(2\Delta_{fg}\diffX{f}\diffX{g} \log \frac{\diffFXi{h}{g}}{\diffFXi{h}{f}}\right)\diff x \wedge \diff y,
$$
where
\begin{align*}
\diffX{f} = \diffFX{x}{f}\diffX{x} &+ \diffFX{y}{f}\diffX{y} = \frac{1}{\Delta_{fg}}\left(g_y\diffX{x} - g_x\diffX{y}\right) = \\ &= \frac{\nu}{\Delta_{fg}}\left(\{z,x\}_{\mathbf q}\diffX{x} + \{z,y\}_{\mathbf q}\diffX{y}\right) = \frac{\nu}{\Delta_{fg}}\{ z, * \}_{\mathbf q},\\
\diffX{g} = \diffFX{x}{g}\diffX{x} &+ \diffFX{y}{g}\diffX{y} = \frac{1}{\Delta_{fg}}\left(-f_y\diffX{x} + f_x\diffX{y}\right) = \\ &= -\frac{\mu}{\Delta_{fg}}\left(\{z,x\}_{\mathbf p}\diffX{x} + \{z,y\}_{\mathbf p}\diffX{y}\right) = -\frac{\mu}{\Delta_{fg}}\{ z, * \}_{\mathbf p}.
\end{align*}
Since $h$ is a Casimir function of $\PP+\QQ$, we have $\{z,h\}_{\mathbf p} + \{z,h\}_{\mathbf q} = 0$, and
$$
\frac{\diffFXi{h}{g}}{\diffFXi{h}{f}} = -\frac{\mu}{\nu} \cdot \frac{\{z,h\}_{\mathbf p}}{\{z,h\}_{\mathbf q}} = \frac{\mu}{\nu}.
$$
Consequently,
$$
\Theta_{xy} = \left(2\Delta_{fg}\diffX{f}\diffX{g} \log \frac{\mu}{\nu}\right)\diff x \wedge \diff y = (L_1 - L_2)\,\diff x \wedge \diff y,
$$
where
\begin{align*}
L_1 = 2\Delta_{fg}\diffX{f}\diffX{g} \log \mu,\\
L_2 = 2\Delta_{fg}\diffX{f}\diffX{g} \log \nu.
\end{align*}
Let us compute $L_1$. We have
$$
\diffX{g} \log \mu = -\frac{\mu}{\Delta_{fg}}\left(\{z,x\}_{\mathbf p}\frac{\mu_x}{\mu} + \{z,y\}_{\mathbf p}\frac{\mu_y}{\mu}\right) = - \frac{1}{\Delta_{fg}}\left(\{z,x\}_{\mathbf p}\,\mu_x + \{z,y\}_{\mathbf p}\,\mu_y\right).
$$
Recall that $\mu\{z,x\}_{\mathbf p} = f_y$ and $\mu\{y,z\}_{\mathbf p} = f_x$, so
\begin{align}\label{integratingFactorDerivative}
\diffX{x}\left(\mu\{z,x\}_{\mathbf p}\right) =  \diffX{y}\left(\mu\{y,z\}_{\mathbf p}\right),
\end{align}
and
\begin{align*}
\{z,x\}_{\mathbf p}\,\mu_x + \{z,y\}_{\mathbf p}\,\mu_y = \mu\left(\diffX{x}\{z,x\}_{\mathbf p} + \diffX{y}\{z,y\}_{\mathbf p}\right) = \mu \cdot \divergence \sgrad{{\mathbf p}}{z}.
\end{align*}
Consequently,
$$
\diffX{g} \log \mu = -\frac{\mu \cdot \divergence \sgrad{{\mathbf p}}{z}}{\Delta_{fg}} = -\frac{\divergence \sgrad{{\mathbf p}}{z}}{\nu\Delta_z},
$$
and
$$
\diffX{f}\diffX{g} \log \mu = -\diffX{f}\frac{\divergence \sgrad{{\mathbf p}}{z}}{\nu\Delta_z} = M_1 + M_2,
$$
where
\begin{align*}
M_1 &= - \frac{\divergence \sgrad{{\mathbf p}}{z}}{\Delta_z} \cdot \diffX{f}{\frac{1}{\nu}},\\
M_2 &=  - \frac{1}{\nu}\diffX{f}\frac{\divergence \sgrad{{\mathbf p}}{z}}{\Delta_z}.	
\end{align*}
We have
$$
M_1 = \frac{\divergence \sgrad{{\mathbf p}}{z}}{\nu^2\Delta_z}\cdot \diffFX{\nu}{f} = \frac{\divergence \sgrad{{\mathbf p}}{z}}{\nu^2\Delta_z}\cdot\frac{\nu}{\Delta_{fg}}\left(\{z,x\}_{\mathbf q}\,\nu_{x} + \{z,y\}_{{\mathbf q}}\,\nu_{y}\right).
$$
Analogously to (\ref{integratingFactorDerivative}), we get
\begin{align}\label{integratingFactorDerivative2}
\diffX{x}\left(\nu\{z,x\}_{\mathbf q}\right) =  \diffX{y}\left(\nu\{y,z\}_{\mathbf q}\right),
\end{align}
and
$$
\{z,x\}_{{\mathbf q}}\,\nu_x + \{z,y\}_{{\mathbf q}}\,\nu_y = \nu \cdot \divergence \sgrad{{\mathbf q}}{z},
$$
so
$$
M_1 = \frac{\divergence \sgrad{{\mathbf p}}{z}}{\nu^2\Delta_z} \cdot \frac{\nu}{\Delta_{fg}} \cdot \nu \cdot \divergence \sgrad{{\mathbf q}}{z} = \frac{\divergence \sgrad{{\mathbf p}}{z}\cdot \divergence \sgrad{{\mathbf q}}{z}}{\Delta_z \Delta_{fg}}.
$$
Further,
$$
M_2 =  - \frac{1}{\nu}\diffX{f}\frac{\divergence \sgrad{{\mathbf q}}{z}}{\Delta_z} = -\frac{1}{\Delta_{fg}}\left\{z,\frac{\divergence \sgrad{{\mathbf p}}{z}}{\Delta_z}\right\}_{\!\!{\mathbf q}}.
$$
So,
$$
L_1 = 2\Delta_{fg}(M_1 + M_2) = 2\cdot\frac{\divergence \sgrad{{\mathbf p}}{z}\cdot \divergence \sgrad{{\mathbf q}}{z}}{\Delta_z} - 2\left\{z,\frac{\divergence \sgrad{{\mathbf p}}{z}}{\Delta_z}\right\}_{\!\!{\mathbf q}}.
$$
To compute $L_2$, we change the order of partial derivatives $\diffXi{f}$, $\diffXi{g}$. We get
\begin{align*}
L_2 = 2\,\frac{\divergence \sgrad{{\mathbf p}}{z}\cdot \divergence \sgrad{{\mathbf q}}{z}}{\Delta_z} - 2\left\{z,\frac{\divergence \sgrad{{\mathbf q}}{z}}{\Delta_z}\right\}_{\!\!{\mathbf p}}.
\end{align*}
So,
$$
\Theta_{xy} = (L_1 - L_2)\diff x \wedge \diff y = 2\left(\,\left\{z,\frac{\divergence \sgrad{{\mathbf q}}{z}}{\Delta_z}\right\}_{\!\!{\mathbf p}} - \left\{z,\frac{\divergence \sgrad{{\mathbf p}}{z}}{\Delta_z}\right\}_{\!\!{\mathbf q}}\,\right) \diff x \wedge \diff y,
$$
q.e.d.

\begin{remark}
	The most important step of the proof is to introduce the integrating factors $\nu, \mu$ and use the compatibility conditions (\ref{integratingFactorDerivative}), (\ref{integratingFactorDerivative2}). All the rest is a straightforward computation. What is not obvious \textit{a priori}, is that the integrating factors $\nu, \mu$, which cannot be found explicitly,  disappear and do not enter the final formula for $\Theta$.\par
Also, it helps a lot to change the order of partial derivatives $\diffXi{f}$, $\diffXi{g}$ when computing $L_2$. Otherwise, second derivatives of $\nu, \mu$ arise, and the computation becomes very hard.
\end{remark}

\section{Proof of Theorem \ref{conThm}}\label{proof2}
Let $\PP,\QQ$ be two generators of the pencil. To prove the existence part, choose vector fields $X,Y,Z$ such that
\begin{align*}
\PP &= X \wedge Y,\\
\QQ &= X \wedge Z.
\end{align*}
Clearly, such vector fields exist (however, in general, they do not define a coordinate system; if they do, the pencil is flat). Vector fields $X,Y$ are tangent to the symplectic leaves of $\PP$, so $[X,Y]$ must be a linear combination of $X,Y$. Analogously, $[X,Z]$ is a linear combination of $X,Z$, and $[X,Y+Z]$ is a linear combination of $X, Y+Z$. So,
\begin{align}\label{commutators}
\begin{aligned}
\vphantom{X}[X,Y] &= aX + cY,\\
[X,Z] &= bX + cZ.
\end{aligned}
\end{align}
Since $\PP$ and $\QQ$ must be covariantly constant, we have
\begin{align*}
\nabla_{\mathbf{w}} X \wedge Y + X \wedge \nabla_{\mathbf{w}}  Y &= 0,\\
\nabla_{\mathbf{w}}  X \wedge Z + X \wedge \nabla_{\mathbf{w}}  Z &= 0
\end{align*}
for any vector field $W$. These conditions can be written as
\begin{align}\label{compCond}
\begin{aligned}
\nabla&_{\mathbf{w}} X = \alpha(W) X,\\
\nabla_{\mathbf{w}} Y &= \beta(W) X -\alpha(W) Y,\\
\nabla_{\mathbf{w}} Z &= \gamma(W) X - \alpha(W) Z,
\end{aligned}
\end{align}
where $\alpha, \beta, \gamma$ are $1$-forms. Any connection given by (\ref{compCond}) is compatible with the pencil. However, it is not necessarily torsion-free. A connection is torsion-free if and only if
$$
[U,V] = \nabla_{\mathbf{u}} V - \nabla_{\mathbf{v}} U
$$
for any vector fields $U,V$. So, (\ref{compCond}) is torsion-free if and only if
\begin{align}\label{commutators2}
\begin{aligned}
\vphantom{X}[X,Y] &= (\beta(X) - \alpha(Y)) X -\alpha(X) Y,\\
[X,Z] &= (\gamma(X) - \alpha(Z)) X - \alpha(X) Z,\\
[Y,Z] = (&\gamma(Y) - \beta(Z)) X +\alpha(Z)Y - \alpha(Y) Z.
\end{aligned}
\end{align}
Conditions (\ref{commutators}) imply that (\ref{commutators2}) is solvable with respect to $\alpha, \beta, \gamma$, which proves the first part of the theorem.\par
The second part of the theorem is proved by writing down both forms in coordinates. To simplify computations, bring $\PP$ to a constant form.
\section*{Acknowledgements}
This work was partially supported by the Dynasty Foundation and by the Russian Foundation for Basic Research (project no. 12-01-31497).

\bibliographystyle{unsrt}
\bibliography{curvature}

\end{document}